\documentclass[a4paper,11pt]{amsart}

\usepackage[T1]{fontenc}
\usepackage[utf8]{inputenc}
\usepackage[english]{babel}
\usepackage{amssymb}
\usepackage{amsmath}
\usepackage{amsthm}
\usepackage{mathtools}
\usepackage{enumitem}
\usepackage[margin=30mm]{geometry}
\usepackage{hyperref}

\newtheorem{thm}{Theorem}[section]
\newtheorem{lemma}[thm]{Lemma}

\newtheorem{prop}[thm]{Proposition}
\theoremstyle{definition}

\theoremstyle{remark}
\newtheorem{remark}[thm]{Remark}

\newtheorem{question}[thm]{Question}

\newcommand{\bF}{\mathbf{F}}
\newcommand{\bN}{\mathbf{N}}
\newcommand{\bZ}{\mathbf{Z}}
\newcommand{\bR}{\mathbf{R}}
\newcommand{\cA}{\mathcal{A}}
\newcommand{\cE}{\mathcal{E}}
\newcommand{\cG}{\mathcal{G}}
\newcommand{\cO}{\mathcal{O}}
\newcommand{\cR}{\mathcal{R}}
\newcommand{\cU}{\mathcal{U}}
\newcommand{\abs}[1]{\lvert#1\rvert}
\newcommand{\acts}{\curvearrowright}
\newcommand{\ind}{\mathbf{1}}
\newcommand{\MALG}{\mathrm{MALG}}
\DeclareMathOperator{\Aut}{Aut}
\DeclareMathOperator{\Sch}{Sch}
\DeclareMathOperator{\id}{id}
\DeclareMathOperator{\supp}{supp}

\usepackage[draft,inline,nomargin,index]{fixme}

\fxsetup{theme=color,mode=multiuser}
\FXRegisterAuthor{kw}{akw}{\color{red}KW}
\definecolor{DarkGreen}{RGB}{30,150,100}
\FXRegisterAuthor{ct}{act}{\color{DarkGreen}CT}

\title{Orbit equivalence and total weak mixing of free group actions}
\author[Konrad Wr\'{o}bel]{Konrad Wr\'{o}bel}
\address[Konrad Wr\'{o}bel]{Department of Mathematics, The University of Texas at Austin, USA}
\email{konrad.wrobel@austin.utexas.edu}

\subjclass[2020]{37A20, 37A25, 28D15}
\begin{document}

\begin{abstract}
We prove that the orbit equivalence class of every free ergodic probability-measure-preserving (pmp) action of a free group contains a totally weak mixing action. 
Equivalently, every ergodic treeable pmp equivalence relation of cost $n\in\bN\cup\{\infty\}$ is generated by a free totally weak mixing action of $\bF_n$. This answers a question of Miller and Tserunyan.

The proof goes by considering a Polish space of edge slidings along a fixed mixing transformation and proving that for every $w\not=e\in\bF_n$ the set of edge slidings that produce an action with $w$ weakly mixing forms a comeager set.
\end{abstract}

\maketitle

\section{Introduction}

Let $G\acts(X,\mu)$ be a pmp action of a countable group. The action is \textbf{weak mixing} if the diagonal action $G\acts (X\times X,\mu\otimes\mu)$ is ergodic, and it is \textbf{totally weak mixing (totally ergodic)} if its restriction to every infinite subgroup of $G$ is weak mixing (ergodic). Mixing actions are totally weak mixing, and totally weak mixing actions are totally ergodic.

Fix $n\in\bN\cup\{\infty\}$ and let $\bF_n$ be the free group on $n$ generators. By Gaboriau's theorem on cost of treeings \cite[Th\'eor\`eme IV.1]{Ga00}, the orbit equivalence relation of a free pmp action of $\bF_n$ is treeable of cost $n$, and Hjorth \cite{Hjorth:CostAttained} proved the converse: every ergodic treeable pmp equivalence relation of cost $n$ is generated by a free pmp action of $\bF_n$. Miller and Tserunyan \cite[Theorem~1.3]{MillerTserunyan} strengthened this by producing such an action in which each of the $n$ free generators acts ergodically, and asked \cite[Question~1.8]{MillerTserunyan} whether one can further ask that every nontrivial group element act ergodically. We answer this in the positive, showing that every nontrivial element can in fact be taken to be weak mixing.

\begin{thm}\label{thm:main}
Let $n\in\bN\cup\{\infty\}$ and let $\alpha:\bF_n\acts(X,\mu)$ be a free ergodic pmp action. Then there is a free totally weak mixing pmp action $\beta:\bF_n\acts(X,\mu)$ with $\cR_\beta=\cR_\alpha$.

Equivalently, every ergodic treeable pmp equivalence relation of cost $n$ is generated by a free totally weak mixing pmp action of $\bF_n$.
\end{thm}

Since $\bF_n$ is torsion-free, an action of $\bF_n$ is totally weak mixing precisely when each nontrivial group element acts by a weakly mixing transformation. In his thesis, Dai \cite[Theorems 4.2.1 and 4.2.9]{Dai:thesis} obtained partial progress towards \cite[Question~1.8]{MillerTserunyan}: given a free ergodic pmp action $\alpha$, he produced a free orbit equivalent action $\beta$ so that $\beta(s_1)$, $\beta(s_2)$ and $\beta(s_1s_2)$ are all ergodic, along with other configurations, such as ensuring $\beta(s_1^2)$, $\beta(s_2^2)$, and $\beta(s_1s_2s_1s_2)$ are all ergodic.

We now briefly describe the strategy. By \cite[Theorem 1.3]{MillerTserunyan} together with Dye's theorem \cite{Dye}, the relation $\cR_\alpha$ is generated by a free action of $\bF_n=\langle s_1,\dots, s_n\rangle$ whose first generator is mixing, so we may assume that $\alpha(s_1)$ is mixing. Keeping $\alpha(s_1)$ and the relation fixed, we may then slide the remaining generators along the orbits of $\alpha(s_1)$: for a tuple $\Phi=(\phi_2,\dots,\phi_n)$ of elements of the full group $[\alpha(s_1)]$, the transformations $\alpha(s_1)$ and $\alpha(s_i)\phi_i$ generate a new free action $\beta_\Phi$ of $\bF_n$ with the same orbits as $\alpha$. This gives a Polish space $[\alpha(s_1)]^{n-1}$ of free actions generating $\cR_\alpha$ and the theorem follows once we prove that for each nontrivial $w\in\bF_n$ the set of $\Phi$ for which $\beta_\Phi(w)$ is weakly mixing is comeager.  As is typical of genericity arguments, all of the difficulty is concentrated in Lemma \ref{lem:density} which proves the density of particular open sets that provide a quantitative approximation to weak mixing of $\beta_\Phi(w)$. 

The genericity aspect of this argument can be compared with several existing results. Most directly, Kerr and Pichot \cite{KerrPichot} showed that weak mixing actions form a comeager subset of $A(\Gamma,X,\mu)$ for every countable discrete group $\Gamma$ without property (T). For $\bR$-flows, Kerr, Li, and Pichot \cite{KerrLiPichot} showed that totally weak mixing flows are comeager. 
In a different direction, Bowen \cite{Bowen:WeakDensity} showed that every orbit equivalence class of a free pmp action of $\bF_n$ is weakly dense in $A(\bF_n,X,\mu)$. These results all concern the weak topology on the space of actions, whereas here the orbit equivalence relation is fixed and the genericity argument takes place in the uniform topology on the space of edge slidings.

We note a natural follow-up question which remains open.
\begin{question}[Tucker-Drob] \label{question:Mixing}
Can $\beta$ in Theorem \ref{thm:main} be taken to be mixing?
\end{question}
This question remains particularly curious in the context that the space of mixing actions admits a Polish topology \cite{Tikhonov}; one might be particularly interested in whether the space of mixing actions remembers everything about orbit equivalence.

\medskip\noindent\textbf{Acknowledgements.} The author  thanks Robin Tucker-Drob for explaining some of the motivation behind Question \ref{question:Mixing} and Clinton Conley for bringing \cite{Dai:thesis} to his attention.

\section{Preliminaries}
\subsection{Weak mixing}
Throughout $(X,\mu)$ is a standard nonatomic probability space and $\MALG$ denotes its measure algebra.

A transformation $T\in \Aut(X,\mu)$ is \textbf{weak mixing} if for every quadruple $A_0,A_1,B_0,B_1\in \MALG$ of positive measure sets there is an $N\ge 1$ with $\mu(T^NA_i\cap B_i)>0$ for $i\in\{0,1\}$. Towards the genericity argument, we require a countable set of (open) conditions that imply weak mixing. 

\begin{prop}\label{prop:approx}
Let $C>0$, let $\cA\subseteq\MALG$ be a countable dense subset of positive measure sets with respect to the metric $\mu(A\triangle B)$, and let $T\in\Aut(X,\mu)$. Suppose that for all $A_0,A_1,B_0,B_1\in\cA$ there is $N\ge1$ with $\mu(T^NA_i\cap B_i)>C\mu(A_i)\mu(B_i)$ for $i\in\{0,1\}$. 

Then $T$ is weak mixing.
\end{prop}

\begin{proof}
Fix positive measure sets $A_0,A_1,B_0,B_1$ and put $\delta:=\min_i\frac{C\mu(A_i)\mu(B_i)}{4(C+1)}$. Choose $A_i',B_i'\in\cA$ with $\mu(A_i\triangle A_i')<\delta$ and $\mu(B_i\triangle B_i')<\delta$. By assumption, there is an $N\ge 1$ so that $\mu(T^NA_i'\cap B_i')>C\mu(A_i')\mu(B_i')$ for $i\in\{0,1\}$. 
Then
\begin{align*}
\mu(T^NA_i\cap B_i)&\ge\mu(T^NA_i'\cap B_i')-2\delta>C\mu(A_i')\mu(B_i')-2\delta\\
&\ge C\mu(A_i)\mu(B_i)-2\delta(C+1)\ge\tfrac12C\mu(A_i)\mu(B_i)>0
\end{align*}
verifying that $T$ is weak mixing.
\end{proof}

An action of a countable group is \textbf{totally weak mixing (totally ergodic)} if its restriction to every infinite subgroup of $G$ is weak mixing (ergodic). This definition agrees with the one given in \cite{Furstenberg:RecurrenceBook,Bergelson:PET,KerrLiPichot} in the category of torsion-free groups. 

\subsection{Full groups and edge slidings}
If $\cR$ is a pmp countable Borel equivalence relation on $(X,\mu)$, then its \textbf{full group} $[\cR]$, given by
\[
[\cR]=\{T\in\Aut(X,\mu):(Tx,x)\in\cR\text{ for a.e.\ }x\},
\]
is a Polish group when equipped with the \textbf{uniform metric} $d_u(S,T)=\mu(\{x:Sx\neq Tx\})$ (see for instance \cite[Proposition 3.2]{book:GlobalAspects}). For $T\in\Aut(X,\mu)$ we write $\cR_T$ for the orbit equivalence relation of $T$ and abbreviate $[T]:=[\cR_T]$.

Fix a free pmp action $\alpha:\bF_n\acts(X,\mu)$. For a tuple $\Phi=(\phi_2,\dots,\phi_n)\in[\alpha(s_1)]^{n-1}$, let $\beta_\Phi:\bF_n\acts(X,\mu)$ be the action determined by $\beta_\Phi(s_1):=\alpha(s_1)$ and $\beta_\Phi(s_i):=\alpha(s_i)\phi_i$ for $2\leq i\leq n$. We equip $[\alpha(s_1)]^{n-1}$ with the product of the uniform topologies, so that it is a Polish space and hence a Baire space.

Let $S$ be the free generating set of $\bF_n$. The \textbf{Schreier graphing} $\Sch(\alpha, S)$ is the labeled Borel graph on $(X,\mu)$ with a directed edge labeled by $s_i$ with origin $x$ and terminus $\alpha(s_i)x$ for each $i$ and almost every $x\in X$.
In the case where the action is a $\bZ$-action with generator $T$, we denote by $\Sch(T)$ the graph with a directed edge with origin $x$ and terminus $Tx$ for almost every $x\in X$.

\begin{lemma}\label{lem:slide}
For every $\Phi\in[\alpha(s_1)]^{n-1}$ the action $\beta_\Phi$ is free and $\cR_{\beta_\Phi}=\cR_\alpha$.
\end{lemma}

\begin{proof}
The Schreier graphing of $\beta_\Phi$ is obtained from that of $\alpha$ by sliding the $s_i$-labeled edge with origin $\phi_i(x)$ and terminus $\alpha(s_i)\phi_i(x)$ to the edge with origin $x$ and the same terminus for each $i\ge2$ and almost every $x\in X$. 
As $\phi_i\in[\alpha(s_1)]$, the two origins lie in a common $\alpha(s_1)$-orbit, so this is a $\Sch(\alpha(s_1))$-based edge slide of $\Sch(\alpha, S)$ in the sense of \cite[Definition~3.6]{MillerTserunyan}. 
Such edge slidings preserve connectivity and acyclicity by \cite[Observation~3.7(a) and Lemma~3.8]{MillerTserunyan}. So the resulting Schreier graphing $\Sch(\beta_\Phi,S)$ is acyclic with the same connected components as $\Sch(\alpha,S)$, finishing the proof. 
\end{proof}

\section{Genericity of weak mixing}
In this section, we fix a free pmp action $\alpha:\bF_n\acts(X,\mu)$ such that $\alpha(s_1)$ is mixing, together with a countable dense family $\cA\subseteq\MALG$ of positive measure sets. For $w\in\bF_n$, $C>0$, and positive measure sets $A_0,A_1,B_0,B_1\subseteq X$, put
\[
\cO_w^C(A_0,A_1,B_0,B_1)=\bigcup_{N\ge1}\bigl\{\Phi\in[\alpha(s_1)]^{n-1}:\mu\bigl(\beta_\Phi(w)^NA_i\cap B_i\bigr)>C\mu(A_i)\mu(B_i)\text{ for }i\in\{0,1\}\bigr\}
\]
and let $\cE_w^C$ be the intersection of the sets $\cO_w^C(A_0,A_1,B_0,B_1)$ over all quadruples from $\cA$. By Proposition \ref{prop:approx}, the transformation $\beta_\Phi(w)$ is weak mixing for every $\Phi\in\cE_w^C$.

\begin{lemma}\label{lem:open}
Each $\cO_w^C(A_0,A_1,B_0,B_1)$ is open, and hence $\cE_w^C$ is a $G_\delta$ set.
\end{lemma}

\begin{proof}
The word $w$ involves finitely many generators, so $\Phi\mapsto\beta_\Phi(w)^N$ is continuous, and $\abs{\mu(SA\cap B)-\mu(S'A\cap B)}\leq d_u(S,S')$ for all $S,S'\in\Aut(X,\mu)$. Each set in the above union is therefore open.
\end{proof}

\begin{lemma}\label{lem:density}
Let $w\in\bF_n$ be a nontrivial cyclically reduced word such that $s_j^{\pm1}$ occurs in $w$ for some $j\neq1$. Then $\cO_w^{C}(A_0,A_1,B_0,B_1)$ is dense in $[\alpha(s_1)]^{n-1}$ for all positive measure sets $A_0,A_1,B_0,B_1\subseteq X$ for some $C=C_w>0$.
\end{lemma}

\begin{proof}
Fix $j$ so that $s_j^{\pm1}$ occurs in $w$ and let $m=m_j\ge 1$ denote the number of occurrences. 
Fix $\Phi=(\phi_i)_{2\le i\le n}\in[\alpha(s_1)]^{n-1}$ and a neighborhood $\cU$ of $\Phi$. Fix $\varepsilon>0$ such that every tuple agreeing with $\Phi$ off the $j$-th coordinate and whose $j$-th coordinate is within $\varepsilon$ of $\phi_j$ in the uniform metric lies in $\cU$.

Write $w=t_\ell\cdots t_1$ as a reduced word in the generators and put $T:=\beta_\Phi(w)$. Let $P$ be the set of $1\le p\le \ell$ with $t_p\in\{s_j,s_j^{-1}\}$. Set $q=\min P$ to be the least element. For $p\in P$,
\[
R_p:=
\begin{cases}\beta_\Phi(t_{p-1}\cdots t_1)& \text{ if }t_p=s_j\\ 
\beta_\Phi(s_j^{-1}t_{p-1}\cdots t_1)&\text{ if }t_p=s_j^{-1}
\end{cases}
\]
so that $R_px$ is the origin of the positively oriented $s_j$-edge traversed at the $p$-th letter of the $w$-path based at $x$. Since $w$ is cyclically reduced, the bi-infinite path labelled by $\cdots www\cdots$ is a geodesic in the Schreier graphing of $\beta_\Phi$ so the positively oriented $s_j$-edges along it are pairwise distinct. As $\beta_\Phi$ is free by Lemma \ref{lem:slide}, this gives $R_q^{-1}R_pT^k=\id$ only for $(p,k)=(q,0)$.

Fix $N\ge1$ with $\frac2{mN}<\varepsilon$. For $L\ge1$ and $i\in\{0,1\}$, let
\[
F_{i,L}:=\sum_{r=0}^{N-1}\big(\ind_{A_i}\circ T^{-r}\big)\big(\ind_{B_i}\circ T^{N-r}R_q^{-1}\alpha(s_1)^LR_q\big).
\]
After the change of variables $y=R_qx$, the integral of the $r$-th summand of $F_{i,L}$ equals $\mu\bigl(R_qT^rA_i\cap \alpha(s_1)^{-L}R_qT^{-(N-r)}B_i\bigr)$, so mixing of $\alpha(s_1)$ gives $\int_XF_{i,L}\,d\mu>\frac N2\mu(A_i)\mu(B_i)$ for both $i$ once $L$ is large enough. Freeness of $\alpha$ makes the transformations $R_q^{-1}\alpha(s_1)^LR_q$ pairwise distinct, so we may in addition require that $R_q^{-1}\alpha(s_1)^{\pm L}R_q$ avoids the finite set $\{R_q^{-1}R_pT^k:p\in P,\ \abs k\leq N\}$. Fix such an $L$ and write $F_i:=F_{i,L}$.

Define an auxiliary pmp (undirected) Borel graph $\cG$ on $(X,\mu)$ that records conflicts in the following manner: a distinct pair $(y,x)\in X^2$ is set to be adjacent in $\cG$ if 
\[
y=R_q^{-1}\alpha(s_1)^{-\epsilon_0L}R_pT^kR_q^{-1}\alpha(s_1)^{\epsilon_1L}R_qx
\]
for some $p\in P$, $|k|\le N$, and $\epsilon_0,\epsilon_1\in\{0,1\}$. These conflicts are chosen mainly so that they cover the set of pairs $(y,x)$ such that if we slide both the $s_j$ edge located at $R_qy$ and at $R_qx$ by $\alpha(s_1)^L$, there is a path of the form $w^N$ starting at some vertex $z$ that follows both altered edges. Simple counting gives that the degree of $\cG$ is uniformly bounded above by $8m(2N+1)<25mN$. By the choice of $L$ and freeness, the transformation $(R_q^{-1}\alpha(s_1)^{-\epsilon_0L}R_q)(R_q^{-1}R_pT^k)(R_q^{-1}\alpha(s_1)^{\epsilon_1L}R_q)$ is the identity only when $p=q$, $k=0$, and $\epsilon_0=\epsilon_1$.

Since $F_0,F_1$ both take values in $\{0,\dots, N\}$, by Hatami--Lov\'asz--Szegedy \cite[Lemma~7.9]{HLS} (see also Ab\'ert--Weiss \cite{AbertWeiss} or Tucker-Drob \cite{TuckerDrob:WeakContainment}), we may find a measurable coloring $c: X\to \{1,\dots, 100mN\}$ such that the joint distribution of the radius-one neighborhoods of $\cG$ as colored by $(F_0,F_1, c)$ is arbitrarily close to the distribution obtained by assigning independent uniform colors in $\{1,\dots, 100mN\}$ in addition to the coloring $(F_0,F_1)$. 
If this approximation is sufficiently close, then the set $D=\{x\in X:c(x)=1 \text{ and }c(y)\not=1 \text{ for every }y\mathrel{\cG} x\}$
is a $\cG$-independent set satisfying $\mu(D)<\frac{1}{100mN}$ and 
\begin{align*}
\int_D F_i\;d\mu&>\frac{1}{100mN}(1-\frac{1}{100mN})^{25mN}\int_XF_i\;d\mu\\
&\ge\frac{1}{100mN}(1-\frac{25mN}{100mN})\int_X F_i\;d\mu>\frac{1}{1000mN}\int_XF_i\;d\mu
\end{align*}
for $i\in \{0,1\}$. Fix such a set $D\subseteq X$.

Independence of $D$ gives that $R_qD$ and $\alpha(s_1)^LR_qD$ are disjoint. Let $V\in[\alpha(s_1)]$ be the involution interchanging $x$ and $\alpha(s_1)^Lx$ for $x\in R_qD$ and fixing every other point. Then $\mu(\supp V)=2\mu(D)<\frac{1}{50mN}<\varepsilon$
so replacing $\phi_j$ by $\phi_jV$ and leaving the remaining coordinates unchanged gives the tuple $\Phi'=(\phi_2,\dots, \phi_jV, \dots, \phi_n)\in\cU$.

It follows that for every $x\in D$ and $0\leq r<N$,
\[
\beta_{\Phi'}(w)^N T^{-r}x
=
T^{N-r}R_q^{-1}\alpha(s_1)^L R_qx
\]
via unwrapping the definition of $V$ and $D$. 
Consider first the old $w^N$-path from $T^{-r}x$ up to $T^{N-r}x$. 
The path passes through $R_qx$ at which point $\beta_{\Phi'}$ now acts via the involution $V$ and follows another $w^N$ path from $\alpha(s_1)^LR_qx$. 
Indeed, every occurrence of an origin of an $s_j$ edge (other than $\alpha(s_1)^LR_qx$) along either path from $T^{-r}x$ to $R_qx$ or the path from $\alpha(s_1)^LR_qx$ to $T^{N-r}R_q^{-1}\alpha(s_1)^LR_qx$ avoids $\supp V=R_qD\cup \alpha(s_1)^LR_qD$ by independence of $D$.

The sets $T^{-r}D$, $0\leq r<N$, are pairwise disjoint by independence of $D$ which implies
\begin{align*}
\mu\big(\beta_{\Phi'}(w)^N A_i\cap B_i\big)&\ge\sum_{r=0}^{N-1}\mu\big(T^rA_i\cap T^r\beta_{\Phi'}(w)^{-N}B_i\cap D\big)\\
&= \int_D \sum_{r=0}^{N-1}\big(\ind_{A_i}\circ T^{-r}\big)\big(\ind_{B_i}\circ T^{N-r}R_q^{-1}\alpha(s_1)^LR_q\big)\;d\mu\\
&> \frac{1}{1000mN}\int_X F_i\;d\mu> \frac{1}{2000m}\mu(A_i)\mu(B_i)
\end{align*}
for $i\in\{0,1\}$ where the last inequality follows from the choice of $L$.
Thus $\Phi'\in\cU\cap\cO_w^{C}(A_0,A_1,B_0,B_1)$ for $C=\frac{1}{2000m}$ completing the proof since $m$ depends only on the choice of $w$ and $j$.
\end{proof}

\begin{proof}[Proof of Theorem \ref{thm:main}]
Let $\alpha:\bF_n\acts(X,\mu)$ be a free ergodic pmp action. By \cite[Theorem 1.3]{MillerTserunyan} and Dye's Theorem \cite{Dye}, we may replace $\alpha$ by an action with the same orbit equivalence relation and assume that $\alpha(s_1)$ is mixing.

Fix $w\in\bF_n\setminus\{e\}$ and write $w=uvu^{-1}$ with $v$ cyclically reduced. Since $\beta_\Phi(w)$ and $\beta_\Phi(v)$ are conjugate in $\Aut(X,\mu)$, the transformation $\beta_\Phi(w)$ is weak mixing if and only if $\beta_\Phi(v)$ is. If $v=s_1^k$ for some $k\neq0$ then $\beta_\Phi(v)=\alpha(s_1)^k$ is mixing for every $\Phi$. Otherwise $v$ involves some generator $s_j^{\pm1}$ with $j\neq1$, and Lemmas \ref{lem:open} and \ref{lem:density} along with Baire category show that $\cE_v^{C_v}$ is a comeager subset of $[\alpha(s_1)]^{n-1}$. By Proposition \ref{prop:approx}, $\beta_\Phi(v)$ is weak mixing for every $\Phi\in \cE_v^{C_v}$. In either case, the set of $\Phi$ so that $\beta_\Phi(w)$ is weak mixing is comeager for every $w\in \bF_n$, and so we may find a $\Phi$ so that $\beta_\Phi$ is totally weak mixing. By Lemma \ref{lem:slide}, $\beta_\Phi$ is a free totally weak mixing pmp action with $\cR_\alpha=\cR_{\beta_\Phi}$, completing the proof.
\end{proof}

\begin{remark}
    The proof above may seem to use Baire category in an essential fashion and certainly the argument lives naturally in that language. However, one can avoid genericity and instead produce an appropriate action by hand using Borel-Cantelli and Lemma \ref{lem:density}, by making perturbations that are summably small.
\end{remark}

\bibliography{references}
\bibliographystyle{alphaurl}

\end{document}